\documentclass{amsart}
\usepackage{amssymb}
\usepackage{latexsym}
\usepackage{amssymb}
\usepackage{amsmath}
\usepackage{mathrsfs}
\usepackage{amsbsy}
\usepackage{cite}
\usepackage{hyperref}
\usepackage[latin1]{inputenc}

\theoremstyle{plain}
\newtheorem{theorem}{Theorem}[section]

\newtheorem{conjecture}[theorem]{Conjecture}
\newtheorem{corollary}[theorem]{Corollary}

\newtheorem{lemma}[theorem]{Lemma}

\newtheorem{proposition}[theorem]{Proposition}

\theoremstyle{definition}

\newtheorem{remark}[theorem]{Remark}
\newtheorem{definition}[theorem]{Definition}

\begin{document}
\title[Strongly productive ultrafilters on semigroups]{Strongly productive
ultrafilters on semigroups}
\author[David Fern\'andez]{David J. Fern\'andez Bret\'on}
\address{Department of Mathematics and Statistics, York University\\
N520 Ross, 4700 Keele Street, M3J 1P3\\
Toronto (Ontario), Canada}
\curraddr{Mathematics Department, University of Michigan \\
2074 East Hall, 530 Church Street \\
Ann Arbor, MI 48109-1043, U.S.A.}
\email{djfernan@umich.edu}
\urladdr{http://www-personal.umich.edu/\textasciitilde djfernan/}
\thanks{}
\author{Martino Lupini}
\address{Department of Mathematics and Statistics, York University\\
N520 Ross, 4700 Keele Street, M3J 1P3\\
Toronto (Ontario), Canada}
\curraddr{Fakultät für Mathematik, Universität Wien, Oskar-Morgenstern-Platz
1, Room 02.126, 1090 Wien, Austria.}
\email{martino.lupini@univie.ac.at}
\urladdr{www.lupini.org}
\thanks{David Fern\'andez was supported by scholarship number 213921/309058
of the Consejo Nacional de Ciencia y Tecnología (CONACyT), Mexico. Martino
Lupini was supported by the York University Elia Scholars Program.}
\date{}
\dedicatory{}

\begin{abstract}
We prove that if $S$ is a commutative semigroup with well-founded universal
semilattice or a solvable inverse semigroup with well-founded semilattice of
idempotents, then every strongly productive ultrafilter on $S$ is
idempotent. Moreover we show that any very strongly productive ultrafilter
on the free semigroup with countably many generators is sparse, answering a
question of Hindman and Legette Jones.%
\end{abstract}

\maketitle











\section{Introduction\label{Section: introduction}}

Let $S$ be a multiplicatively denoted semigroup. If $\vec{x}=\left(
x_{n}\right) _{n\in \omega }$ is a sequence of elements of $S$, then the 
\emph{finite products set} associated with $\vec{x}$, denoted by $%
\mathop{\mathrm{FP}}\nolimits(\vec{x})$, is the set of products (taken in
increasing order of indices)%
\begin{equation*}
\prod_{i\in a}x_{i}\in S
\end{equation*}%
where $a$ ranges among the finite subsets of $\omega $. If $k\in \omega $
then $\mathop{\mathrm{FP}}\nolimits_{k}(\vec{x})$ stands for the \textup{FP}%
-set associated with the sequence $\left( x_{n+k}\right) _{n\in \omega }$. A
subset of $S$ is called an \textup{FP}-set if it is of the form $%
\mathop{\mathrm{FP}}\nolimits(\vec{x})$ for some sequence $\vec{x}$ in $S$,
and an \textup{IP}-set if it contains an \textup{FP}-set (see \cite[%
Definition 16.3]{Hindman-Strauss}). An ultrafilter $p$ on $S$ (a gentle
introduction to ultrafilters can be found in \cite[Appendix B]%
{Capraro-Lupini}) is \emph{strongly productive} as in \cite[Section 1]%
{Hindman-Jones}) if it has a basis of \textup{FP}-sets. This means that for
every $A\in p$ there is an \textup{FP}-set contained in $A$ that is a member
of $p$. When $S$ is an additively denoted commutative semigroup, then the
finite products sets are called \emph{finite sums sets }or \textup{FS}-sets
and denoted by $\mathop{\mathrm{FS}}\nolimits(\vec{x})$. Moreover strongly
productive ultrafilters are called in this context \emph{strongly summable}
(see \cite[Definition 1.1]{Hin-Prot-Strauss}).

The concept of strongly summable ultrafilter was first considered in the
case of the semigroup of positive integers in \cite{Hindman-summable} by
Hindman upon suggestion of van Douwen (see also the notes at the end of \cite%
[Chapter 12]{Hindman-Strauss}). 
Later Hindman, Protasov, and Strauss studied in \cite{Hin-Prot-Strauss}
strongly summable ultrafilters on arbitrary abelian groups. Theorem 2.3 in 
\cite{Hin-Prot-Strauss} asserts that any strongly summable ultrafilter on an
abelian group $G$ is idempotent, i.e.\ an idempotent element of the
semigroup compactification $\beta G$ of $G$, which can be seen as the
collection of all ultrafilters on $G$. Similarly, strongly productive
ultrafilters on a free semigroup are also idempotent by \cite[Lemma 2.3]%
{Hindman-Jones}.

In this paper we provide a common generalization of these results to a class
of semigroups containing in particular all commutative semigroups with well-founded universal semilattice, and solvable inverse semigroups with well-founded semilattice of idempotents. (The notion of universal semilattice of
a semigroup is presented in \cite[Section III.2]{Grillet}. Inverse
semigroups are introduced in \cite[Section II.2]{Grillet}, and within these
the class of solvable semigroups is defined in \cite[Definition 3.2]{Piochi1}%
. Solvable groups are the solvable inverse semigroups with exactly one
idempotent element by \cite[Theorem 3.4]{Piochi1}.)

\begin{theorem}
\label{Theorem: idempotent}If $S$ is either a commutative semigroup with
well-founded universal semilattice, or a solvable inverse semigroup with
well-founded semilattice of idempotents, then every strongly productive
ultrafilter on $S$ is idempotent.
\end{theorem}

In order to prove Theorem \ref{Theorem: idempotent} we find it convenient to
consider the following strengthening of the notion of strongly productive 
ultrafilter:

\begin{definition}
\label{Definition: superstrongly}A nonprincipal strongly productive
ultrafilter $p$ on a semigroup $S$ is \emph{regular }if it contains an
element $B$ with the following property: Whenever $\vec{x}$ is a sequence in 
$S$ such that $\mathop{\mathrm{FP}}\nolimits(\vec{x})\subset B$, the set $%
x_{0}\mathop{\mathrm{FP}}\nolimits_{1}(\vec{x})$ does not belong to $p$.
\end{definition}

\begin{remark}
\label{Remark: inverse image}Suppose that $S,T$ are semigroups, $%
f:S\rightarrow T$ is a semigroup homomorphism, and $p$ is a strongly
productive ultrafilter on $S$. Denote by $q$ the ultrafilter on $T$ defined
by $B\in q$ if and only if $f^{-1}\left[ B\right] \in p$ (note that $q$ is
the image of $p$ under the unique extension of $f$ to a continuous function
from the \v Cech-Stone compactification of $S$ to the \v Cech-Stone
compactification of $T$). It is easy to see that $q$ is a strongly productive 
ultrafilter on $T$. Moreover if $q$ is regular then $p$ is regular.
\end{remark}

We will show that the notions of strongly productive and regular strongly
productive ultrafilter coincide for the classes of semigroups considered in
Theorem \ref{Theorem: idempotent}.

\begin{theorem}
\label{Theorem: superstrongly}If $S$ is either a commutative semigroup with
well-founded universal semilattice, or a solvable inverse semigroup with
well-founded semilattice of idempotents, then every nonprincipal strongly
productive ultrafilter on $S$ is regular.
\end{theorem}

The \emph{universal semilattice} of a semigroup $S$ is the quotient of $S$
by the smallest semilattice congruence $\mathcal{N}$ on $S$, see \cite[%
Section III.2]{Grillet}. When $S$ is a commutative inverse semigroup the set 
$E(S)$ of idempotent elements of $S$ is a semilattice, and the restriction
to $E(S)$ of the quotient map from $S$ to $S\left/ \mathcal{N}\right. $ is
an isomorphism from $E(S)$ onto $S\left/ \mathcal{N}\right. $. Recall also 
that an ordered set is said to be \emph{well-founded} if
every nonempty subset has a least element.

Although not using this terminology, there is a well-known argument showing
that any regular strongly productive ultrafilter is idempotent, see for
example \cite[Theorem 12.19]{Hindman-Strauss}, or \cite[Lemma 2.3]%
{Hindman-Jones}. We reproduce the argument in Lemma \ref{Lemma:
superstrongly} below for convenience of the reader. Using this fact, Theorem %
\ref{Theorem: idempotent} will be a direct consequence of Theorem \ref%
{Theorem: superstrongly}.

\begin{lemma}
\label{Lemma: superstrongly}Suppose that $S$ is a semigroup, and $p$ is an
ultrafilter on $S$. If $p$ is regular strongly productive, then $p$ is
idempotent.
\end{lemma}

\begin{proof}
Fix an element $B$ of $p$ witnessing the fact that $p$ is regular. Suppose
that $A$ is an element of $p$ and $\vec{x}$ is a sequence in $S$ such that $%
\mathop{\mathrm{FP}}\nolimits(\vec{x})\subset A\cap B$. Since 
\begin{equation*}
\mathop{\mathrm{FP}}\nolimits(\vec{x})=\{x_{0}\}\cup x_{0}%
\mathop{\mathrm{FP}}\nolimits_{1}(\vec{x})\cup \mathop{\mathrm{FP}}%
\nolimits_{1}(\vec{x}),
\end{equation*}%
and $p$ is nonprincipal and regular, it follows that $\mathop{\mathrm{FP}}%
\nolimits_{1}(\vec{x})\in p$. Using this argument, one can show by induction
that $\mathop{\mathrm{FP}}\nolimits_{n}(\vec{x})\in p$ for every $n\in
\omega $. Now notice that, if $x=\prod_{i\in a}x_{i}\in \mathop{\mathrm{FP}}%
\nolimits(\vec{x})$ and $n=\max (a)+1$, then $x\mathop{\mathrm{FP}}%
\nolimits_{n}(\vec{x})\subseteq \mathop{\mathrm{FP}}\nolimits(\vec{x}%
)\subseteq A$. Therefore $\mathop{\mathrm{FP}}\nolimits_{n}(\vec{x}%
)\subseteq x^{-1}A$ and since the former set is an element of $p$, so is the
latter. Hence 
\begin{equation*}
\mathop{\mathrm{FP}}\nolimits(\vec{x})\subset \left\{ x\in S:x^{-1}A\in
p\right\}
\end{equation*}%
and so the latter set belongs to $p$. This shows that $p$ is idempotent$.$
\end{proof}


To our knowledge it is currently not known if the existence of a semigroup $%
S $ and a nonprincipal strongly productive ultrafilter on $S$ that is not
idempotent is consistent with the usual axioms of set theory. We think that
Theorem \ref{Theorem: idempotent} as well as \cite[Theorem 2.3]%
{Hin-Prot-Strauss} provide evidence that for all semigroups $S$, every
strongly productive ultrafilter on $S$ should be idempotent.

\begin{conjecture}
\label{Conjecture: all regular}Every strongly productive ultrafilter on an
arbitrary semigroup is idempotent.
\end{conjecture}

This paper is organized as follows: In Section \ref{Section: regular} we
introduce the notion of \textup{IP}-regular (partial) semigroup and observe
that \textup{IP}-regular semigroups satisfy the conclusion of Theorem \ref%
{Theorem: superstrongly}. In Section \ref{Section: abelian} we show that
commutative cancellative semigroups are \textup{IP}-regular. In Section \ref%
{Section: IP-regular} we record some closure properties of the class of 
\textup{IP}-regular semigroups, implying in particular that all (virtually)
solvable groups are \textup{IP}-regular. In Section \ref{Section: main
theorem} we present the proof of Theorem \ref{Theorem: superstrongly}.
Finally in Section \ref{Section: sparseness} we discuss sparseness of
strongly productive ultrafilters, and show that every very strongly productive
ultrafilter on the free semigroup with countably many generators is sparse,
answering Question 3.8 from \cite{Hindman-Jones}.%
\subsection*{Acknowledgments} We would like to thanks the anonymous referee and Jimmie Lawson for their comments and remarks.

\section{\textup{IP}-regularity\label{Section: regular}}

A \emph{partial semigroup }as defined in \cite[Section I.3]{Grillet-abelian}
is a set $P$ endowed with a partially defined (multiplicatively denoted)
operation such that for every $a,b,c\in P$ 
\begin{equation*}
(ab)c=a(bc)
\end{equation*}%
whenever both $(ab)c$ and $a(bc)$ are defined. Observe that in particular
every semigroup is a partial semigroup. Moreover any subset of a partial
semigroup is naturally endowed with a partial semigroup structure when one
considers the restriction of the operation. An element $a$ of a partial
semigroup $P$ is idempotent if $a\cdot a$ is defined and equal to $a$. The
set of idempotent elements of $P$ is denoted by $E(P)$. The notion of 
\textup{FP}-set and \textup{IP}-set admit straightforward generalizations to
the framework of partial semigroups. If $\vec{x}$ is a sequence of elements
of a partial semigroup $P$ such that all the products (taken in increasing
order of indices)%
\begin{equation*}
\prod_{i\in a}x_{i}
\end{equation*}%
where $a$ is a finite subset of $\omega $ are defined, then the \textup{FP}%
-set $\mathop{\mathrm{FP}}\nolimits(\vec{x})$ is the set of all such
products. A subset of $P$ is an \textup{IP}-set if it contains an \textup{FP}%
-set. Analogous to the case of semigroups, an ultrafilter $p$ on a partial
semigroup $P$ is \emph{strongly productive} if it is has a basis of \textup{%
FP}-sets. A strongly productive ultrafilter is \emph{regular} if it contains
a set $B$ with the following property: If $\vec{x}$ is a sequence in $P$
such that all finite products from $\vec{x}$ are defined and belong to $B$,
then $x_{0}\mathop{\mathrm{FP}}\nolimits_{1}(\vec{x})$ does not belong to $p$%
.

\begin{definition}
\label{Definition: strongly regular}A partial semigroup $P$ is \emph{%
strongly \textup{IP}-regular} if for every sequence $\vec{x}$ in $P$ such
that all the finite products from $\vec{x}$ are defined the set%
\begin{equation*}
x_{0}\mathop{\mathrm{FP}}\nolimits\nolimits_{1}(\vec{x})
\end{equation*}%
is not an \textup{IP}-set.
\end{definition}

Since a subset of a partial semigroup is still a partial semigroup, we can
speak about strongly \textup{IP}-regular subsets of a partial semigroups,
which are just strongly \textup{IP}-regular partial semigroups with respect
to the induced partial semigroup structure.

\begin{definition}
\label{Definition: regular}A partial semigroup $P$ is \emph{\textup{IP}%
-regular} if $P\left\backslash E\left( P\right) \right. $ is the union of
finitely many strongly \textup{IP}-regular sets.
\end{definition}

It is immediate from the definition that a partial semigroup is \textup{IP}%
-regular whenever it is the union of finitely many \textup{IP}-regular
subsets (i.e.\ subsets which are \textup{IP}-regular partial semigroups with
respect to the induced partial semigroup structure).

\begin{remark}
\label{Remark: subsemigroup}Any subset of a (strongly) \textup{IP}-regular
partial semigroup is (strongly) \textup{IP}-regular. Any finite partial
semigroup is strongly \textup{IP}-regular.
\end{remark}

We will show in Section \ref{Section: abelian} that every cancellative
commutative semigroup is \textup{IP}-regular. The relevance of the notion of 
\textup{IP}-regularity stems from the fact that any \textup{IP}-regular
group satisfies the conclusion of Theorem \ref{Theorem: superstrongly}. The
same is in fact true for any \textup{IP}-regular partial semigroup with
finitely many idempotent elements.

\begin{lemma}
\label{Lemma: IP-regular superstrongly}If $S$ is an \textup{IP}-regular
partial semigroup with finitely many idempotent elements, then every
nonprincipal strongly productive ultrafilter on $S$ is regular.
\end{lemma}

\begin{proof}
Suppose that $p$ is a strongly summable ultrafilter on $S$. Since $p$ is
nonprincipal, $S\left\backslash E(S)\right. \in p$. Given that $S$ is 
\textup{IP}-regular, the ultrafilter $p$ must have a strongly \textup{IP}%
-regular member $A$. It is clear that $A$ witnesses the fact that $p$ is
regular.
\end{proof}

The class of \textup{IP}-regular partial semigroups has interesting closure
properties. We have already observed that a subset of an \textup{IP}-regular
partial semigroup is \textup{IP}-regular. Moreover the inverse image of an 
\textup{IP}-regular partial semigroup under a partial homomorphism is 
\textup{IP}-regular. Recall that a \emph{partial homomorphism }from a
partial semigroup $P$ to a partial semigroup $Q$ is a function $%
f:P\rightarrow Q$ such that for every $a,b\in P$ such that $ab $ is defined, 
$f(a)f(b)$ is defined and $f(ab)=f(a)f(b)$

\begin{lemma}
\label{Lemma: inverse image}Suppose that $f:P\rightarrow Q$ is a partial
homomorphism. If $Q$ is \textup{IP}-regular and $f^{-1}\left[ E\left(
Q\right) \right] $ is \textup{IP}-regular, then $P$ is \textup{IP}-regular.
\end{lemma}

\begin{proof}
Observe that the image of an \textup{IP}-set under a partial homomorphism is
an \textup{IP}-set. It follows that $f^{-1}\left[ B\right] $ is strongly 
\textup{IP}-regular whenever $B\subset Q\left\backslash E\left( Q\right)
\right. $ is strongly \textup{IP}-regular. The fact that $P$ is \textup{IP}%
-regular follows easily from these observations.
\end{proof}

In particular Lemma \ref{Lemma: inverse image} guarantees that the extension
of an \textup{IP}-regular group by an \textup{IP}-regular group is \textup{IP%
}-regular. More generally one can consider groups admitting a subnormal
series with \textup{IP}-regular factor groups. Recall that a subnormal
series of a group $G$ is a finite sequence $(a)$ of subgroups of $G$ such
that $A_{0}=\left\{ 1_{G}\right\} $, $A_{n}=G$, and $A_{i}\unlhd A_{i+1}$
for every $i\in n$. The quotients $A_{i+1}\left/ A_{i}\right. $ for $i\in n$
are called \emph{factor groups }of the series. Proposition \ref{Proposition:
subnormal series IP} can be easily obtained from Lemma \ref{Lemma: inverse
image} by induction on the length of the subnormal series.

\begin{proposition}
\label{Proposition: subnormal series IP}Suppose that $G$ is a group, and $%
(A_i)_{i\in n} $ is a subnormal series of $G$. If the factor groups $%
A_{i+1}\left/ A_{i}\right. $ are \textup{IP}-regular for all $i\in n$, then $%
G$ is \textup{IP}-regular.
\end{proposition}

A consequence of Proposition \ref{Proposition: subnormal series IP} is that
solvable groups are \textup{IP}-regular. This follows from the facts that
solvable groups are exactly those that admit a subnormal series where all
the factor groups are abelian, and that all abelian groups are \textup{IP}%
-regular. The latter fact will be proved in the following section.

\section{Cancellative commutative semigroups\label{Section: abelian}}

Throughout this section all (partial) semigroups are \emph{additively
denoted }and assumed to be \emph{cancellative }and \emph{commutative}.

\begin{proposition}
\label{Proposition: abelian}Every cancellative commutative semigroup is 
\textup{IP}-regular (as in Definition \ref{Definition: regular}).
\end{proposition}

A key role in the proof of Proposition \ref{Proposition: abelian} is played
by the notion of rank function.

\begin{definition}
\label{Definition: rank function}A \emph{rank function} on a cancellative
commutative partial semigroup $P$ is a function $\rho$ from $P$ to a well-ordered set with the property that if $\vec{x}$ is a sequence in $P$ such
that all finite sums from $\vec{x}$ are defined, then the two following
conditions are satisfied:

\begin{enumerate}
\item The restriction of $\rho$ to the range $\{x_n\big|n\in\omega\}$ of the
sequence $\vec{x}$ is a finite-to-one function, and

\item if $\rho \left( x_{n}\right) \geq \rho (x_{0})$ for every $n\in \omega 
$ then $x_{0}+\mathop{\mathrm{FS}}\nolimits_{1}(\vec{x})$ is not an \textup{%
IP}-set.
\end{enumerate}
\end{definition}

Remark \ref{Remark: particular rank function} provides an example of a rank
function.

\begin{remark}
\label{Remark: particular rank function}If $\rho $ is a function from $P$ to
a well order such that $\rho \left( x+y\right) =\min \left\{ \rho (x),\rho
(y)\right\} $ and $\rho (x)\neq \rho (y)$ whenever $x,y\in P$ and $x+y$ is
defined, then $\rho $ is a rank function on $A$.
\end{remark}

The relevance of rank functions for the proof of Proposition \ref%
{Proposition: abelian} is stated in Lemma \ref{Lemma: rank function}. Recall
that the family of \textup{IP}-sets of a (partial) semigroup $P$ is
partition regular (see \cite[Corollary 5.15]{Hindman-Strauss}). This means
that if $\mathfrak{F} $ is a finite family of subsets of $P$ and $\bigcup 
\mathfrak{F}$ is an \textup{IP}-set, then $\mathfrak{F}$ contains an \textup{%
IP}-set. This fact will be used in the proof of Lemmas \ref{Lemma: rank
function} and \ref{Lemma: C1}

\begin{lemma}
\label{Lemma: rank function}If there is a rank function on a partial
semigroup $P$, then $P$ is strongly \textup{IP}-regular.
\end{lemma}

\begin{proof}
Fix a rank function $\rho $ from $P$ to a well-ordered set. Suppose that $%
\vec{x}$ is a sequence in $P$ such that all the finite sums are defined. We
claim that $x_{0}+\mathop{\mathrm{FS}}\nolimits_{1}(\vec{x})$ is not an 
\textup{IP}-set. Since $\rho$ restricted to the range of $\vec{x}$ is
finite-to-one, it is possible to pick a permutation $\sigma $ of $\omega $
such that%
\begin{equation*}
\rho \left( x_{\sigma \left( n\right) }\right) \leq \rho \left( x_{\sigma
\left( m\right) }\right)
\end{equation*}%
for every $n\in m\in \omega $. Define $y_{n}=x_{\sigma \left( n\right) }$
and $\vec{y}=\left( y_{n}\right) _{n\in \omega }$. Observe that if $%
x_{0}=y_{n_{0}}$ then%
\begin{eqnarray*}
&&x_{0}+\mathop{\mathrm{FS}}\nolimits\nolimits_{1}(\vec{x}) \\
&=&\left( y_{n_{0}}+\mathop{\mathrm{FS}}\nolimits\nolimits_{n_{0}+1}(\vec{y}%
)\right) \cup \left( y_{n_{0}}+\mathop{\mathrm{FS}}\nolimits\left( \left(
y_{i}\right) _{i=0}^{n_{0}-1}\right) \right) \cup \\
&&\cup\left( y_{n_{0}}+\mathop{\mathrm{FS}}\nolimits\left( \left(
y_{i}\right) _{i=0}^{n_{0}-1}\right) +\mathop{\mathrm{FS}}%
\nolimits\nolimits_{n_{0}+1}(\vec{y})\right) \text{.}
\end{eqnarray*}%
Applying the hypothesis that $\rho $ is a rank function to the sequence $(y_{%
{n_0}+k})_{k\in\omega}$, it follows that%
\begin{equation*}
y_{n_{0}}+\mathop{\mathrm{FS}}\nolimits\nolimits_{n_{0}+1}(\vec{y})
\end{equation*}%
is not an \textup{IP}-set. Now for every $y=\sum_{i\in a}y_i\in %
\mathop{\mathrm{FS}}\nolimits\left( \left( y_{i}\right) _{i\in n_{0}}\right) 
$, if $m=\min(a)$ then we have that 
\begin{equation*}
y_{n_{0}}+y+\mathop{\mathrm{FS}}\nolimits\nolimits_{n_{0}+1}(\vec{y})\subset
y_m+\mathop{\mathrm{FS}}\nolimits_{m+1}(\vec{y})
\end{equation*}%
and since the latter is not an \textup{IP}-set (by applying the fact that $%
\rho$ is a rank function to the sequence $(y_{m+k})_{k\in\omega}$), neither
is the former. Finally%
\begin{equation*}
y_{n_{0}}+\mathop{\mathrm{FS}}\nolimits\left( \left( y_{i}\right) _{i\in
n_{0}}\right)
\end{equation*}%
is finite and hence not an \textup{IP}-set. This allows one to conclude that%
\begin{equation*}
x_{0}+\mathop{\mathrm{FS}}\nolimits\nolimits_{1}(\vec{x})
\end{equation*}%
is not an \textup{IP}-set, as claimed.
\end{proof}

Denote in the following by $\mathbb{R}\left/ \mathbb{Z}\right. $ the
quotient of $\mathbb{R}$ by the subgroup $\mathbb{Z}$. It is a well known
fact that any commutative cancellative semigroup embeds into an abelian
group, see \cite[Proposition II.3.2]{Grillet-abelian}. Moreover any abelian
group can be embedded in a divisible abelian group, and a divisible abelian
group in turn can be embedded into a direct sum of copies of $\mathbb{R}%
\left/ \mathbb{Z}\right. $ (see for example \cite[Theorems 24.1 and 23.1]%
{Fuchs}). It follows that every cancellative commutative semigroup is a
subsemigroup of a direct sum $\left( \mathbb{R}\left/ \mathbb{Z}\right.
\right) ^{\oplus \kappa }$ of $\kappa $ copies of $\mathbb{R}\left/ \mathbb{Z%
}\right. $ for some cardinal $\kappa $. Therefore it is enough to prove
Proposition \ref{Proposition: abelian} for $\left( 
\mathbb{R}
\left/ 
\mathbb{Z}
\right. \right) ^{\oplus \kappa }$. The proof of this fact will occupy the
rest of this section.

Let us fix a cardinal $\kappa $. If $x\in \left( 
\mathbb{R}
\left/ 
\mathbb{Z}
\right. \right) ^{\oplus \kappa }$ and $\alpha\in \kappa $ then $\pi
_{\alpha}(x)$ denotes the $\alpha$-th coordinate of $x$. Elements of $%
\mathbb{R}
\left/ 
\mathbb{Z}
\right. $ will be freely identified with their representatives in $\mathbb{R}
$ (thus we might write something like $t\neq0$, and this really means $%
t\notin\mathbb{Z}$), and if we need to specify a particular representative,
we will choose the unique such in $\left[ 0,1\right) $. Consider the
partition%
\begin{equation*}
\left( 
\mathbb{R}
\left/ 
\mathbb{Z}
\right. \right) ^{\oplus \kappa }=C\cup B\cup \left\{ 0\right\}
\end{equation*}%
of $\left( 
\mathbb{R}
\left/ 
\mathbb{Z}
\right. \right) ^{\oplus \kappa }$, where $B$ is the set of elements of $%
\left( 
\mathbb{R}
\left/ 
\mathbb{Z}
\right. \right) ^{\oplus \kappa }$ of order $2$.

\begin{lemma}
The subset $B$ of $\left( 
\mathbb{R}
\left/ 
\mathbb{Z}
\right. \right) ^{\oplus \kappa }$ is strongly \textup{IP}-regular.
\end{lemma}

\begin{proof}
Observe that $B\cup \left\{ 0\right\} $ is a subgroup of $\left( 
\mathbb{R}
\left/ 
\mathbb{Z}
\right. \right) ^{\oplus \kappa }$ isomorphic to the direct sum of $\kappa $
copies of $%
\mathbb{Z}
\left/ 2%
\mathbb{Z}
\right. $. Thus $B\cup \left\{ 0\right\} $ has the structure of $\kappa $%
-dimensional vector space over $%
\mathbb{Z}
\left/ 2%
\mathbb{Z}
\right. $. If $\vec{x}$ is a sequence in $B$ then $\mathop{\mathrm{FS}}%
\nolimits(\vec{x})\cup \left\{ 0\right\} $ is the vector space generated by $%
\vec{x}$. Moreover the sequence $\vec{x}$ is linearly independent if and
only if $0\notin \mathop{\mathrm{FS}}\nolimits(\vec{x})$ (see \cite[%
Proposition 4.1]{David-2}). Thus if $\vec{x}$ is a sequence in $B$ such that 
$\mathop{\mathrm{FS}}\nolimits(\vec{x})\subset B$ then $\vec{x}$ is a
linearly independent sequence, and hence any element of $\mathop{\mathrm{FS}}%
\nolimits(\vec{x})$ can be written in a unique way as a sum of elements of
the sequence $\vec{x}$. In particular $x_0+\mathop{\mathrm{FS}}\nolimits_1(%
\vec{x})$ consists of those finite sums $\sum_{i\in a}x_i$ such that $0\in a$%
. Thus if $a,b$ are finite subsets of $\omega\left\backslash 1\right. $,
then so is $a\bigtriangleup b$, hence%
\begin{equation*}
x_{0}+\sum_{i\in a}x_{i}+x_{0}+\sum_{i\in b}x_{i}=\sum_{i\in a\bigtriangleup
b}x_{i}\notin x_{0}+\mathop{\mathrm{FS}}\nolimits\nolimits_{1}(\vec{x})\text{%
.}
\end{equation*}%
This shows that whenever $x,y\in x_{0}+\mathop{\mathrm{FS}}\nolimits_{1}(%
\vec{x})$ then $x+y\notin x_0+\mathop{\mathrm{FS}}\nolimits_1(\vec{x})$,
which implies that $x_0+\mathop{\mathrm{FS}}\nolimits_1(\vec{x})$ is not an 
\textup{IP}-set and $B$ is \textup{IP}-regular.
\end{proof}

It remains to show now that $C$ is \textup{IP}-regular. Elements $x\in C$
have order strictly greater than $2$, thus there is at least one $%
\alpha<\kappa$ such that $\pi_\alpha(x)\notin\left\{0,\frac{1}{2}\right\}$,
hence it is possible to define the function $\mu :C\rightarrow \kappa $ by%
\begin{equation*}
\mu (x)=\min \left\{ \alpha\in \kappa :\pi _{\alpha}(x)\notin \left\{ 0,%
\frac{1}{2}\right\} \right\} \text{.}
\end{equation*}%
Consider%
\begin{eqnarray*}
C_{1} &=&\left\{ x\in C:\pi _{\mu (x)}(x)=\frac{1}{4}\right\} \text{;} \\
C_{3} &=&\left\{ x\in C:\pi _{\mu (x)}(x)=\frac{3}{4}\right\} \text{;} \\
C_{2} &=&\left\{ x\in C:\pi _{\mu (x)}(x)\notin \left\{ \frac{1}{4},\frac{3}{%
4}\right\} \right\} \text{.}
\end{eqnarray*}%
Observe that%
\begin{equation*}
C=C_{1}\cup C_{2}\cup C_{3}
\end{equation*}%
is a partition of $C$.

\begin{lemma}
\label{Lemma: C1}The function $\mu $ restricted to $C_{1}$ is a rank
function on $C_{1}$ as in Definition \ref{Definition: rank function}.
\end{lemma}

\begin{proof}
Suppose that $\vec{x}$ is a sequence in $\left( 
\mathbb{R}
\left/ 
\mathbb{Z}
\right. \right) ^{\oplus \kappa }$ such that $\mathop{\mathrm{FS}}\nolimits(%
\vec{x})\subset C_{1}$. We will show that the function $\mu$ restricted to $%
\{x_n\big|n\in\omega\}$ is at most two-to-one, in particular finite-to-one.
This is because if $n,m,k\in\omega$ are three distinct numbers such that $%
\mu(x_n)=\mu(x_m)=\mu(x_k)=\alpha$, then for $\beta<\alpha$ we get that $%
\pi_\beta(x_n+x_m+x_k)$ is an element of $\left\{0,\frac{1}{2}\right\}$
because so are $\pi_\beta(x_n),\pi_\beta(x_m),\pi_\beta(x_k)$. On the other
hand, $\pi_\alpha(x_n+x_m+x_k)=\frac{3}{4}$, which shows that $%
\mu(x_n+x_m+x_k)=\alpha$ but $x_n+x_m+x_k\in C_3$, a contradiction.

Now assume also that $\alpha =\mu (x_{0})\leq \mu \left( x_{i}\right) $ for
every $i\in \omega $. By the previous paragraph, there is at most one $n\in
\omega \left\backslash 1\right. $ such that $\mu \left( x_{n}\right) =\mu
(x_{0})=\alpha $. Thus the first case is when there is such $n$. The first
thing to notice is that for each $k\in \omega \setminus \{0,n\}$, we have
that $\pi _{\alpha }(x_{k})=0$. This is because otherwise, since $\mu
(x_{k})>\alpha $ we would have that $\pi _{\alpha }(x_{k})=\frac{1}{2}$ and
so $\pi _{\alpha }(x_{0}+x_{k})=\frac{3}{4}$. Therefore by an argument
similar to that in the previous paragraph, we would also have $\mu
(x_{0}+x_{k})=\alpha $, which would imply that $x_{0}+x_{k}\in C_{3}$, a
contradiction. Now write 
\begin{equation*}
x_{0}+\mathop{\mathrm{FS}}\nolimits_{1}(\vec{x})=\{x_{0}+x_{n}\}\cup \left(
x_{0}+\mathop{\mathrm{FS}}\nolimits((x_{k})_{k\in \omega \setminus
\{0,n\}})\right) \cup \left( x_{0}+x_{n}+\mathop{\mathrm{FS}}%
\nolimits((x_{k})_{k\in \omega \setminus \{0,n\}})\right)
\end{equation*}%
Clearly $\{x_{0}+x_{1}\}$ is not an \textup{IP}-set, as it is finite. Now
since $\pi _{\alpha }(x_{k})=0$ for $i\notin \omega \setminus \{0,n\}$, it
follows that every element $x\in x_{0}+\mathop{\mathrm{FS}}%
\nolimits((x_{k})_{k\in \omega \setminus \{0,n\}})$ must satisfy $\pi
_{\alpha }(x)=\frac{1}{4}$, which implies that $x_{0}+\mathop{\mathrm{FS}}%
\nolimits((x_{k})_{k\in \omega \setminus \{0,n\}})$ cannot contain the sum
of any two of its elements and consequently it is not an \textup{IP}-set.
Similarly, every element $x\in x_{0}+x_{n}+\mathop{\mathrm{FS}}%
\nolimits((x_{k})_{k\in \omega \setminus \{0,n\}})$ satisfies $\pi _{\alpha
}(x)=\frac{1}{2}$, so this set is, by the same argument, not an \textup{IP}%
-set. Hence $x_{0}+\mathop{\mathrm{FS}}\nolimits_{1}(\vec{x})$ is not an 
\textup{IP}-set.

Now if there is no such $n$, i.e.\ if $\mu(x_k)>\mu(x_0)=\alpha$ for all $%
k>0 $, then arguing as in the previous paragraph we get that $%
\pi_\alpha(x_k)=0$ for all $k>0$. Hence every element $x\in x_0+%
\mathop{\mathrm{FS}}\nolimits_1(\vec{x})$ satisfies that $\pi_\alpha(x)=%
\frac{1}{4}$, therefore the set $x_0+\mathop{\mathrm{FS}}\nolimits_1(\vec{x}%
) $ cannot be an \textup{IP}-set. This concludes the proof that $\mu $ is a
rank function on $C_{1}$.
\end{proof}

Considering the fact that the function $t\mapsto -t$ is an automorphism of $%
\left( 
\mathbb{R}
\left/ 
\mathbb{Z}
\right. \right) ^{\oplus \kappa }$ mapping $C_{1}$ onto $C_{3}$ and
preserving $\mu $ allows one to deduce from Lemma \ref{Lemma: C1} that $\mu $
is a rank function on $C_{3}$ as well. Thus it only remains to show that $%
C_{2}$ is \textup{IP}-regular.

Define%
\begin{equation*}
Q_{i,j}=\left\{ x\in C_{2}:\pi _{\mu (x)}(x)\in \bigcup_{m\in \omega }\left[ 
\frac{i}{4}+\frac{1}{2^{3m+j+3}},\frac{i}{4}+\frac{1}{2^{3m+j+2}}\right)
\right\}
\end{equation*}%
for $i\in \left\{ 0,1,2,3\right\} $ and $j\in \left\{ 0,1,2\right\} $.
Observe that%
\begin{equation*}
C_{2}=\bigcup_{i\in 4}\bigcup_{j\in 3}Q_{i,j}
\end{equation*}%
is a partition of $C_{2}$. In order to conclude the proof of Proposition \ref%
{Proposition: abelian} it is now enough to show that for every $i\in 4$ and $%
j\in 3$ the set $Q_{i,j}$ is \textup{IP}-regular. This will follow from
Lemma \ref{Lemma: Q} by Lemma \ref{Lemma: rank function}.

\begin{lemma}
\label{Lemma: Q}Consider $\kappa \times \omega $ well-ordered by the
lexicographic order. The function $\rho :Q_{i,j}\rightarrow \kappa \times
\omega $ defined by $\rho (x)=\left( \mu (x),m\right) $ where $m$ is the
unique element of $\omega $ such that 
\begin{equation*}
\pi _{\mu (x)}(x)\in \left[ \frac{i}{4}+\frac{1}{2^{3m+j+3}},\frac{i}{4}+%
\frac{1}{2^{3m+j+2}}\right)
\end{equation*}%
is a rank function on $Q_{i,j}$.
\end{lemma}

\begin{proof}
To simplify the notation let us run the proof in the case when $i=j=0$. The
proof in the other cases is analogous. By Remark \ref{Remark: particular
rank function} it is enough to show that if $x$ and $y$ are such that $%
x,y,x+y\in Q_{0,0}$ then $\rho (x)\neq \rho (y)$ and $\rho \left( x+y\right)
=\min \left\{ \rho (x),\rho (y)\right\} $, so suppose that $x$,$y$ are
elements of $Q_{0,0}$ such that $x+y\in Q_{0,0}$ and assume by contradiction
that $\rho (x)=\rho (y)=\left(\alpha,m\right) $. Thus%
\begin{equation*}
\pi _{\alpha}(x),\pi _{\alpha}(y)\in \left[ \frac{1}{2^{3m+3}},\frac{1}{%
2^{3m+2}}\right)
\end{equation*}%
and hence%
\begin{equation*}
\pi _{\alpha}\left( x+y\right) \in \left[ \frac{1}{2^{3m+2}},\frac{1}{%
2^{3m+1}}\right) \text{.}
\end{equation*}%
If $m=0$ then%
\begin{equation*}
\pi _{\alpha}\left( x+y\right) \in \left[ \frac{1}{4},\frac{1}{2}\right)
\end{equation*}%
thus 
\begin{equation*}
x+y\in C_1\cup Q_{1,0}\cup Q_{1,1}\cup Q_{1,3}\text{.}
\end{equation*}%
If $m>0$ then%
\begin{equation*}
\pi _{\alpha}\left( x+y\right) \in \left[ \frac{1}{2^{3\left( m-1\right)
+2+3}},\frac{1}{2^{3\left( m-1\right) +2+2}}\right)
\end{equation*}%
and therefore 
\begin{equation*}
x+y\in Q_{0,2}\text{.}
\end{equation*}%
In either case one obtains a contradiction from the assumption that $x+y\in
Q_{0,0}$. This concludes the proof that $\rho (x)\neq \rho (y)$. We now
claim that $\rho \left( x+y\right) =\min \left\{ \rho (x),\rho (y)\right\} $%
. Define $\rho (x)=\left(\alpha,m\right) $ and $\rho
(y)=\left(\beta,n\right) $. Let us first consider the case when $%
\alpha=\beta $ and without loss of generality $m>n $. In this case%
\begin{equation*}
\pi _{\xi}\left( x+y\right) \in \left\{0,\frac{1}{2}\right\}
\end{equation*}%
for $\xi<\alpha$, while%
\begin{equation*}
\pi _{\alpha}\left( x+y\right) \in \left[ \frac{1}{2^{3m+3}}+\frac{1}{%
2^{3n+3}},\frac{1}{2^{3m+2}}+\frac{1}{2^{3n+2}}\right)
\end{equation*}%
where%
\begin{equation*}
\frac{1}{2^{3m+2}}+\frac{1}{2^{3n+2}}<\frac{1}{2^{3n+1}}<\frac{1}{2^{3\left(
n-1\right) +3}}\text{.}
\end{equation*}%
This shows that $\rho \left( x+y\right) =\left(\alpha,n\right) =\min \left\{
\rho (x),\rho (y)\right\} $. Let us now consider the case when $%
\alpha\neq\beta$ and without loss of generality $\alpha>\beta$. In this case 
\begin{equation*}
\pi _{\xi}\left( x+y\right) \in \left\{ 0,\frac{1}{2}\right\}
\end{equation*}%
for $\xi<\beta$ while%
\begin{equation*}
\pi _{\beta}(x)=0
\end{equation*}%
(because if not then $\pi_\beta(x)=\frac{1}{2}$ and that would imply that $%
x+y\in\bigcup_{j\in3}Q_{1,j}$), and hence%
\begin{equation*}
\pi _{\beta}\left( x+y\right) =\pi _{\beta}(y)\text{.}
\end{equation*}%
This shows that $\rho \left( x+y\right) =\left(\beta,n\right) =\min \left\{
\rho (x),\rho (y)\right\} $. This concludes the proof of the fact that $\rho 
$ satisfies the hypothesis of Remark \ref{Remark: particular rank function}
and, hence, it is a rank function on $Q_{0,0}$.
\end{proof}

\section{The class of \textup{IP}-regular semigroups\label{Section:
IP-regular}}

In this section all semigroups will be denoted multiplicatively. Let us
define $\mathcal{R}$ to be the class of all \textup{IP}-regular semigroups.
Observe that by Proposition \ref{Proposition: abelian} $\mathcal{R}$
contains all commutative cancellative semigroups. We will now show that $%
\mathcal{R}$ contains all Archimedean commutative semigroups. Recall that a
commutative semigroup $S$ is \emph{Archimedean }if for every $a,b\in S$
there is a natural number $n$ and an element $t$ of $S$ such that $a^{n}=bt$%
, see \cite[Section III.1]{Grillet-abelian}. By \cite[Proposition III.1.3]%
{Grillet-abelian} an Archimedean commutative semigroup contains at most one
idempotent.

\begin{proposition}
\label{Proposition: archimedean}Archimedean commutative semigroups are 
\textup{IP}-regular.
\end{proposition}

\begin{proof}
Suppose that $S$ is a commutative Archimedean semigroup. Let us first assume
that $S$ has no idempotent elements: In this case by \cite[Proposition IV.4.1%
]{Grillet} there is a congruence $\mathcal{C}$ on $S$ such that the quotient 
$S\left/ \mathcal{C}\right. $ is a commutative cancellative semigroup with
no idempotent elements. It follows from Proposition \ref{Proposition:
abelian} that $S\left/ \mathcal{C}\right. $ is \textup{IP}-regular, and
therefore $S$ is \textup{IP}-regular by Lemma \ref{Lemma: inverse image}.
Let us consider now the case when $S$ has a (necessarily unique) idempotent
element $e$. Denote by $H_{e}$ the maximal subgroup of $S$ containing $e$.
By \cite[Proposition IV.2.3]{Grillet} $H_{e}$ is an ideal of $S$ and the
quotient $S\left/ H_{e}\right. $ is a commutative nilsemigroup, i.e.\ a
commutative semigroup with a zero element such that every element is
nilpotent. By Lemma \ref{Lemma: inverse image} and Proposition \ref%
{Proposition: abelian} it is therefore enough to show that a commutative
nilsemigroup $T$ is \textup{IP}-regular. Denote by $0$ the zero element of $%
T $. We claim that $T\left\backslash \left\{ 0\right\} \right. $ is strongly 
\textup{IP}-regular (as in Definition \ref{Definition: strongly regular}).
In fact if $\vec{x}$ is a sequence in $T$ such that $\mathop{\mathrm{FP}}%
\nolimits(\vec{x})$ does not contain $0$, then $x_{0}\mathop{\mathrm{FP}}%
\nolimits(\vec{x})$ is not an \textup{IP}-set since $x_{0}$ is nilpotent.
This concludes the proof that $T$ is \textup{IP}-regular, and hence also the
proof of Proposition \ref{Proposition: archimedean}.{}
\end{proof}

Let us now comment on the closure properties of the class $\mathcal{R}$. By
Remark \ref{Remark: subsemigroup}, $\mathcal{R}$ is closed with respect to
taking subsemigroups, and contains all finite semigroups. Moreover by
Proposition \ref{Proposition: subnormal series IP} if a group $G$ has a
subnormal series with factor groups in $\mathcal{R}$, then $G$ belongs to $%
\mathcal{R}$. In particular $\mathcal{R}$ contains all virtually solvable
groups and their subgroups. Proposition \ref{Proposition: free product}
shows that free products of elements of $\mathcal{R}$ with no idempotent
elements are still in $\mathcal{R}$.

\begin{proposition}
\label{Proposition: free product}Suppose that $S,T$ are semigroups. If both $%
S$ and $T$ are \textup{IP}-regular, and $T$ has no idempotent elements, then
the free product $S\ast T$ is \textup{IP}-regular.
\end{proposition}

\begin{proof}
Denote by $T_{1}$ the semigroup obtained from $T$ adding an identity element 
$1$. Consider the semigroup homomorphism from $S\ast T$ to $T_{1}$ sending a
word $w$ to $1$ if $w$ does not contain any letters from $T$, and otherwise
sending $w$ to the element of $T$ obtained from $w$ by erasing the letters
from $S$ and then taking the product in $T$ of the remaining letters of $w$.
Observe that $f^{-1}\left[ \left\{ 1\right\} \right] $ is isomorphic to $S$
and therefore \textup{IP}-regular. The conclusion now follows from Lemma \ref%
{Lemma: inverse image}.
\end{proof}

The particular case of Proposition \ref{Proposition: free product} when $S=T=%
\mathbb{N}$ lets us obtain that the free semigroup on $2$ generators is 
\textup{IP}-regular. Considering the function assigning to a word its
length, which is a semigroup homomorphism onto $\mathbb{N}$, one can see
that a free semigroup in any number of generators is \textup{IP}-regular,
since so is $\mathbb{N}$ . Via Lemma \ref{Lemma: IP-regular superstrongly}
this observation gives a short proof of \cite[Lemma 2.3]{Hindman-Jones}.

\section{The main theorem\label{Section: main theorem}}

We will now present the proof of Theorem \ref{Theorem: superstrongly}.
Suppose that $S$ is a commutative semigroup with well-founded universal
semilattice. Denote by $\mathcal{N}$ the smallest semilattice congruence on $%
S$ as in \cite[Proposition III.2.1]{Grillet}. Recall that the universal
semilattice of $S$ is the quotient $S\left/ \mathcal{N}\right. $ by \cite[%
Proposition III.2.2]{Grillet}. Moreover by \cite[Theorem III.1.2]%
{Grillet-abelian} every $\mathcal{N}$-equivalence class is an Archimedean
subsemigroup of $S$ known as an \emph{Archimedean component} of $S$. Pick a
nonprincipal strongly summable ultrafilter $p$ on $S$. If $p$ contains some
Archimedean component of $S$, then $p$ is regular by Lemma \ref{Lemma:
IP-regular superstrongly} and Proposition \ref{Proposition: archimedean}.
Let us then assume that $p$ does not contain any Archimedean component.
Denote by $f:S\rightarrow S\left/ \mathcal{N}\right. $ the canonical
quotient map, and by $q$ the ultrafilter on $S\left/ \mathcal{N}\right. $
defined by $B\in q$ if and only if $f^{-1}\left[ B\right]\in p $ . By Remark %
\ref{Remark: inverse image}, $q$ is a nonprincipal strongly productive
ultrafilter on $S\left/ \mathcal{N}\right. $, and moreover in order to
conclude that $p$ is regular it is enough to show that $q$ is regular. This
will follow from Lemma \ref{Lemma: semilattice}.

\begin{lemma}
\label{Lemma: semilattice}If $\Lambda $ is a well-founded semilattice, then
any nonprincipal strongly summable ultrafilter $q$ on $\Lambda $ is regular.
\end{lemma}

\begin{proof}
We can assume without loss of generality that $\Lambda $ has a maximum
element $x_{\max }$. For $x\in \Lambda $, denote by $\mathop{\mathrm{pred}}%
\nolimits(x)$ the set 
\begin{equation*}
\left\{ y\in \Lambda :y\leq x\text{ and }y\neq x\right\}
\end{equation*}%
of \textit{strict} predecessors of $x$. We will show by well-founded
induction that, for every $x\in \Lambda $, if $\mathop{\mathrm{pred}}%
\nolimits(x)\in q$ then $q$ is regular. The conclusion will follow from the
observation that $\mathop{\mathrm{pred}}%
\nolimits\left( x_{\max }\right) \in q$. Suppose that $x$ is an
element of $\Lambda $ such that, for every $y\in \mathop{\mathrm{pred}}%
\nolimits(x)$, if $\mathop{\mathrm{pred}}\nolimits(y)\in q$ then $q$ is
regular. Suppose that $\mathop{\mathrm{pred}}\nolimits(x)\in q$. If $%
\mathop{\mathrm{pred}}\nolimits(x)$ witnesses the fact that $q$ is regular,
then this concludes the proof. Otherwise there is a sequence $\vec{y}$ in $%
\Lambda $ such that $\mathop{\mathrm{FP}}\nolimits(\vec{y})\subset %
\mathop{\mathrm{pred}}\nolimits(x)$ and $y_{0}\mathop{\mathrm{FP}}%
\nolimits_{1}(\vec{y})\in p$. Observing that $y_{0}\mathop{\mathrm{FP}}%
\nolimits_{1}(\vec{y})\subset \mathop{\mathrm{pred}}\nolimits(y_{0})\cup
\left\{ y_{0}\right\} $ allows one to conclude that $\mathop{\mathrm{pred}}%
\nolimits(y_{0})\in q$, where $y_{0}\in \mathop{\mathrm{pred}}\nolimits(x)$.
Thus by inductive hypothesis $q$ is regular.
\end{proof}

This concludes the proof of the fact that a nonprincipal strongly summable
ultrafilter on a commutative semigroup with well-founded universal
semilattice is regular. We will now show that the same fact holds for
solvable inverse semigroups with well-founded semilattice of idempotents. An
introduction to inverse semigroups can be found in \cite[Chapter VII]%
{Grillet} or in the monograph \cite{Lawson}. Recall that the semilattice of
idempotents of a commutative inverse semigroup is isomorphic to its
universal semilattice. The notion of solvable inverse semigroup has been
introduced by Piochi in \cite{Piochi1} as a generalization of the notion of
solvable group to the context of inverse semigroups (solvable groups are
thus exactly the solvable inverse semigroups with only one idempotent, see 
\cite[Theorem 3.4]{Piochi1}). Observe that by definition a solvable inverse
semigroup $S$ of class $n+1$ has a commutative congruence $\gamma _{S}$ such
that, if $f:S\rightarrow S\left/ \gamma _{S}\right. $ is the canonical
quotient map, then%
\begin{equation*}
f^{-1}\left[ E\left( S\left/ \gamma _{S}\right. \right) \right]
\end{equation*}%
is an inverse subsemigroup of $S$ of solvability class $n$. Moreover the
solvable inverse semigroups of solvability class $1$ are exactly the
commutative semigroups. The fact that solvable inverse semigroups with
well-founded semilattice of idempotents satisfy the conclusion of Theorem %
\ref{Theorem: superstrongly} will then follow from Remark \ref{Remark:
inverse image} by induction on the solvability class, after we observe that
a homorphic image of a semigroup with well-founded semilattice of
idempotents also has a well-founded semilattice of idempotents. This is the
content of Lemma \ref{Lemma: inverse well founded}.

\begin{lemma}
\label{Lemma: inverse well founded}Suppose that $S$, $T$ are semigroups, and 
$f:S\rightarrow T$ is a surjective semigroup homomorphism. If $S$ is an
inverse semigroup with well-founded semilattice of idempotents, then so is $T$.
\end{lemma}

\begin{proof}
By \cite[Theorem 7.32]{Clifford-Preston-II} $T$ is an inverse semigroup. If $%
B$ is a nonempty subset of the idempotent semilattice $E\left( T\right) $ of 
$T$, let $A$ be the set of idempotent elements $a$ of $S$ such that $f(a)\in
B$. Since by hypothesis the idempotent semilattice $E(S)$ of $S$ is well-founded, $A$ has a minimal element $a_{0}$. We claim that $b_{0}=f(a_{0})$
is a minimal element of $B$. Suppose that $b\in B$ is such that $b\leq b_{0}$%
. By \cite[Chapter 1, Proposition 21(3)]{Lawson} there exists $a\in A$ such
that $f(a)=b\leq b_{0}=f(a_{0})$. Hence by \cite[Chapter 1, Proposition 21(7)%
]{Lawson} there exists $a^{\prime }\in A$ such that $a^{\prime }\leq a_{0}$
and $f(a^{\prime })=f(a)=b$. Since $a$ is a minimal element of $A$ we have
that $a^{\prime }=a_{0}$ and hence $b=f(a^{\prime })=f(a_{0})=b_{0}$. This
concludes the proof that $B$ has a minimal element, and that $E\left(
T\right) $ is well-founded.
\end{proof}

\section{Sparseness\label{Section: sparseness}}

A strongly productive ultrafilter $p$ on a (multiplicatively denoted)
semigroup $S$ is \emph{sparse }(see \cite[Definition 3.9]{Hindman-Jones}) if
for every $A\in p$ there are a sequence $\vec{x}=\left( x_{n}\right) _{n\in
\omega }$ in $S$ and a subsequence $\vec{y}=\left( x_{k_{n}}\right) _{n\in
\omega }$ of $\vec{x}$ such that:

\begin{itemize}
\item $\mathop{\mathrm{FP}}\nolimits(\vec{y})\in p$;

\item $\mathop{\mathrm{FP}}\nolimits(\vec{x})\subset A$;

\item $\left\{ k_{n}:n\in \omega \right\} $ is coinfinite in $\omega $.
\end{itemize}

Suppose that $\mathbb{F}$ is the partial semigroup of finite nonempty
subsets of $\omega $, where, for $a,b\in \mathbb{F}$ the product $ab$ is
defined and equal to $a\cup b$ if and only if $\max(a)<\min(b)$. A strongly
productive ultrafilter on the partial semigroup $\mathbb{F}$ is an \emph{%
ordered union ultrafilter} as defined in \cite[page 92]{Blass}. A strongly
productive ultrafilter $p$ on a multiplicatively denoted semigroup $S$ is 
\emph{multiplicatively isomorphic to an ordered union ultrafilter }if there
is a sequence $\vec{x}$ such that the function%
\begin{eqnarray*}
f:&\mathbb{F}\rightarrow &\mathop{\mathrm{FP}}\nolimits(\vec{x}) \\
&a\mapsto &\prod_{i\in a}x_{i}
\end{eqnarray*}%
is injective, and furthermore 
\begin{equation*}
\left\{ f^{-1}[A]:A\in p\right\}
\end{equation*}%
is an ordered union ultrafilter.

\begin{lemma}
\label{Lemma: mult isom}If $p$ is multiplicatively isomorphic to an ordered
union ultrafilter, then $p$ is sparse strongly productive. In particular
every ordered union ultrafilter is sparse.
\end{lemma}

\begin{proof}
Suppose that the sequence $\vec{x}$ in $S$ and the function $f:\mathbb{F}%
\rightarrow \mathop{\mathrm{FP}}\nolimits(\vec{x})$ witness the fact that $p 
$ is multiplicatively isomorphic to an ordered union ultrafilter. Fix an
element $B$ of $p$, and observe that 
\begin{equation*}
q=\left\{ f^{-1}\left[ A\cap B\right] :A\in p\right\}
\end{equation*}%
is an ordered union ultrafilter. Therefore there is a sequence $\vec{b}$ in $%
\mathbb{F}$ such that all the products from $\vec{b}$ are defined
(equivalently, $\max(b_{i})<\min(b_{i+1})$ for every $i\in \omega $), and $%
\mathop{\mathrm{FP}}\nolimits(\vec{b})\in q$. Moreover by \cite[Theorem 4]%
{Krautzberger} (see also \cite[Theorem 2.6]{Hindman-Steprans-Strauss}) there
is an element $W$ of $q$ contained in $\mathop{\mathrm{FP}}\nolimits(\vec{b}%
) $ such that $\bigcup W$ has infinite complement in $\bigcup_{i\in \omega
}b_{i}$. Denote by $D$ the set of $i\in \omega $ such that $b_{i}\subset
\bigcup W$. Observe that%
\begin{equation*}
\bigcup_{i\in D}b_{i}=\bigcup W
\end{equation*}%
and 
\begin{equation*}
W\subset \mathop{\mathrm{FP}}\nolimits\left( \left( b_{i}\right) _{i\in
D}\right) \text{.}
\end{equation*}%
In particular $D$ has infinite complement in $\omega $ and $%
\mathop{\mathrm{FP}}\nolimits\left( (b_i)_{i\in D}\right) $ belongs to $q$.
Therefore the sequence $\overrightarrow{x}$ in $S$ such that $x_{i}=f(b_{i})$
for every $i\in \omega $ is such that $\mathop{\mathrm{FP}}\nolimits(\vec{x}%
)\subset B$ and $\mathop{\mathrm{FP}}\nolimits\left( \left( x_{i}\right)
_{i\in D}\right) \in p$, witnessing the fact that $p$ is sparse strongly
productive.
\end{proof}

We will now define a condition on sequences that ensures the existence of a
multiplicative isomorphism with an ordered union ultrafilter. This can be
seen as a noncommutative analogue of the notion of \emph{strong uniqueness
of finite sums} introduced in \cite[Definition 3.1]{Hindman-Steprans-Strauss}
in a commutative context.

\begin{definition}
\label{Definition: uniqueness} A sequence $\vec{x}$ in a semigroup $S$
satisfies the \emph{ordered uniqueness of finite products} if the function%
\begin{eqnarray*}
f:\mathbb{F}&\rightarrow  &\mathop{\mathrm{FP}}\nolimits(\vec{x}) \\
a&\mapsto& \prod_{i\in a}x_{i}
\end{eqnarray*}%
is an isomorphism of partial semigroups from $\mathbb{F}$ to $%
\mathop{\mathrm{FP}}\nolimits(\vec{x})$. Equivalently $f$ is injective and
if $a$,$b$ are elements of $\mathbb{F}$ such that $f(a)f(b)\in %
\mathop{\mathrm{FP}}\nolimits(\vec{x})$, then the maximum element of $a$ is
strictly smaller than the minimum element of $b$.
\end{definition}

For example suppose that $S$ is the free semigroup on countably many
generators $\left\{ s_{n}:n\in \omega \right\} $. It is not difficult to see
that the sequence $\left( s_{n}\right) _{n\in \omega }$ in $S$ satisfies the
ordered uniqueness of finite products.
.

\begin{remark}
\label{Remark: ordedness isomorphism}If a strongly productive ultrafilter $p$
on $S$ contains $\mathop{\mathrm{FP}}\nolimits(\vec{x})$ for some sequence $%
\vec{x}$ in $S$ satisfying the ordered uniqueness of finite products, then $%
p $ is multiplicatively isomorphic to an ordered union ultrafilter.
\end{remark}

Remark \ref{Remark: ordedness isomorphism} follows immediately from the fact
that an ordered union ultrafilter is just a strongly productive ultrafilter
on the partial semigroup $\mathbb{F}$.

The following immediate consequence of Remark \ref{Remark: ordedness
isomorphism} and Lemma \ref{Lemma: mult isom} can be seen as a
noncommutative analogue of \cite[Theorem 3.2]{Hindman-Steprans-Strauss} (see
also \cite[Corollary 2.9]{David-2}).

\begin{corollary}
\label{Corollary: uniqueness sparse} Let $p$ be a strongly productive
ultrafilter on a semigroup $S$. If $p$ contains $\mathop{\mathrm{FP}}%
\nolimits(\vec{x})$ for some sequence $\vec{x}$ satisfying the ordered
uniqueness of finite products, then $p$ is sparse.
\end{corollary}

In the remainder of this section, we will present an application of
Corollary \ref{Corollary: uniqueness sparse} to a question of Neil Hindman
and Lakeshia Legette Jones from \cite{Hindman-Jones} about \emph{very
strongly productive }ultrafilters on the free semigroup on countably many
generators.

Recall that a sequence $\vec{y}$ on a semigroup $S$ is a \emph{product
subsystem }of the sequence $\vec{x}$ in $S$ if there is a sequence $\left(
a_{n}\right) _{n\in \omega }$ in $\mathbb{F}$ such that $y_{n}=\prod_{i\in
a_{n}}x_{i}$ and the maximum element of $a_{n}$ is strictly smaller than the
minimum element of $a_{n+1}$ for every $n\in \omega $. Suppose that $S$ is
the free semigroup on countably many generators, and $\vec{s}$ is an
enumeration of its generators. A \emph{very strongly productive ultrafilter }%
on $S$ as in \cite[Definition 1.2]{Hindman-Jones} is an ultrafilter $p$ on $%
S $ generated by sets of the form $\mathop{\mathrm{FP}}\nolimits(\vec{x})$
where $\vec{x}$ is a product subsystem of $\vec{s}$.

\begin{theorem}
\label{Theorem: very sparse} Every very strongly productive ultrafilter on
the free semigroup $S$ is multiplicatively isomorphic to an ordered union
ultrafilter, and hence sparse.
\end{theorem}

\begin{proof}
Observe that by \cite[Theorem 4.2]{Hindman-Jones} very strongly productive
ultrafilters on $S$ are exactly the strongly productive ultrafilters
containing $\mathop{\mathrm{FP}}\nolimits(\vec{s})$ as an element. In
particular, since the sequence $\vec{s}$ satisfies the ordered uniqueness of
finite products, all very strongly productive ultrafilters on $S$ are
multiplicatively isomorphic to ordered union ultrafilters by Remark \ref%
{Remark: ordedness isomorphism}, and hence sparse by Lemma \ref{Lemma: mult
isom}.
\end{proof}

Theorem \ref{Theorem: very sparse} answers Question 3.26 from \cite%
{Hindman-Jones}. Corollary 3.11 of \cite{Hindman-Jones} asserts that a
sparse very strongly productive ultrafilter on $S$ can be written only
trivially as a product of ultrafilters on the free group on the same
generators. Since by Theorem \ref{Theorem: very sparse} any very strongly
productive ultrafilter on $S$ is sparse, one can conclude that the
conclusion of \cite[Corollary 3.11]{Hindman-Jones} holds for any very
strongly productive ultrafilter on $S$. This is the content of Corollary \ref%
{Corollary: very sparse}.

\begin{corollary}
\label{Corollary: very sparse}Let $G$ be the free group on the sequence of
generators $\vec{s}$, and let $S$ be the free semigroup on the same
generators. Suppose that $p$ is a very strongly productive ultrafilter on $S$%
. If $q$,$r$ are ultrafilters on $G$ such that $qr=p$, then there is an
element $w$ of $G$ such that one of the following statements hold:

\begin{enumerate}
\item $r=wp$ and $q=pw^{-1}$;

\item $r=w$ and $q=pw^{-1}$;

\item $r=wp$ and $q=w^{-1}$.
\end{enumerate}

In particular, if $q,r\in G^{\ast }$ are such that $qr=p$, then $r=wp$ and $%
q=pw^{-1}$ for some $w\in G$.
\end{corollary}

\end{document}